\documentclass{amsart}
\usepackage{enumitem}
\usepackage{amssymb}
\usepackage{amsthm}
\usepackage{amsmath}
\usepackage{physics}
\usepackage{tikz}
\usepackage{mathdots}
\usepackage{yhmath}
\usepackage{cancel}
\usepackage{color}
\usepackage{siunitx}
\usepackage{array}
\usepackage{multirow}
\usepackage{gensymb}
\usepackage{tabularx}
\usepackage{extarrows}
\usepackage{booktabs}
\usetikzlibrary{fadings}
\usetikzlibrary{patterns}
\usetikzlibrary{shadows.blur}
\usetikzlibrary{shapes}

\usepackage{enumerate}
\usepackage[scr=boondox]{mathalfa}

\usepackage{hyperref}

\theoremstyle{plain}
\newtheorem{proposition}{Proposition}[section]
\newtheorem{theorem}[proposition]{Theorem}
\newtheorem{lemma}[proposition]{Lemma}
\newtheorem{corollary}[proposition]{Corollary}
\theoremstyle{definition}
\newtheorem{example}[proposition]{Example}
\newtheorem{definition}[proposition]{Definition}

\theoremstyle{remark}
\newtheorem{remark}[proposition]{Remark}

\newtheorem{question}[proposition]{Question}

\newtheorem*{claim}{Claim}

\DeclareMathOperator{\Aut}{Aut}

\DeclareMathOperator{\Cay}{Cay}
\DeclareMathOperator{\SL}{\mathsf{SL}}
\DeclareMathOperator{\GL}{\mathsf{GL}}

\DeclareMathOperator{\PSL}{\mathsf{PSL}}

\DeclareMathOperator{\Hom}{Hom}
\DeclareMathOperator{\Out}{Out}

\DeclareMathOperator{\Gr}{Gr} 
\DeclareMathOperator{\id}{id} 
\DeclareMathOperator{\Id}{Id}

\DeclareMathOperator{\Mod}{Mod}

\DeclareMathOperator{\Isom}{Isom}
\DeclareMathOperator{\Inn}{Inn}

\DeclareMathOperator{\Sp}{Sp}

\DeclareMathOperator{\SO}{SO}

\DeclareMathOperator{\Stab}{\mathrm{Stab}}

\DeclareMathOperator{\Fc}{\mathcal{F}}

\DeclareMathOperator{\Cb}{\mathbb{C}}

\DeclareMathOperator{\Hb}{\mathbb{H}}

\DeclareMathOperator{\Kb}{\mathbb{K}}
\DeclareMathOperator{\Nb}{\mathbb{N}}

\DeclareMathOperator{\Rb}{\mathbb{R}}

\DeclareMathOperator{\Zb}{\mathbb{Z}}

\newcommand{\equaldef}{\overset{\mathrm{def}}{=}}

\begin{document}

\title{Simple Anosov representations of closed surface groups}
\author{Nicolas Tholozan}
\address{\'Ecole Normale Sup\'erieure PSL, CNRS}
\email{nicolas.tholozan@ens.fr}

\author{Tianqi Wang}
\address{National University of Singapore}
\email{twang@u.nus.edu}
\thanks{Wang was partially supported by the NUS-MOE grant R-146-000-270-133 and by the Merlion PhD program 2021}

\date{\today}

\begin{abstract}
We introduce and study \emph{simple Anosov representations} of closed hyperbolic surface groups, analogous to Minsky's \emph{primitive stable representations} of free groups. We prove that the set of simple Anosov representations into $\SL(d,\Cb)$ with $d \geqslant 4$ strictly contains the set of Anosov representations. As a consequence, we construct domains of discontinuity for the mapping class group action on character varieties which contain non-discrete representations.
\end{abstract}

\maketitle

\setcounter{tocdepth}{1}
\tableofcontents

\section{Introduction}

Given a finitely generated group $\Gamma$ and a complex linear group $G$, the group of outer automorphisms $\Out(\Gamma)$ of $\Gamma$ acts on the \emph{character variety} $\chi(\Gamma,G)$, the GIT quotient of $\Hom(\Gamma,G)$ under the conjugation action of $G$, by precomposition. This action is of primordial interest in various topics, such as the study of locally homogeneous geometric structures on manifolds, or isomonodromic deformations of complex differential equations.

When $\Gamma$ has a large automorphism group (e.g. when $\Gamma$ is a free group or a surface group), the action of $\Out(\Gamma)$ on character varieties can be very chaotic, and a first interesting question is whether one can construct large domains of discontinuity for this action. A broad family of examples have been produced by the theory of \emph{Anosov representations} of hyperbolic groups, see \cite{Lab}. Anosov representations are quasi-isometrically embedded (they are equivalent for $\SL(2,\Cb)$) and stable under small deformations. They thus form open domains of character varieties on which $\Out(\Gamma)$ acts properly discontinuously (see \cite{LabourieMCGAction}).

However, these domains of discontinuity are not necessarily maximal. In \cite{Min}, Minsky constructed examples of so-called \emph{primitive stable representations} of a free group into $\SL(2,\Cb)$, and proved that they form an open domain of discontinuity which contains strictly the set of quasi-isometric embeddings. Minsky's construction has been generalized to higher rank by Kim--Kim \cite{KK} and by the second author in \cite{W23} who developed the notion of \emph{primitive Anosov representation}.

Roughly speaking, primitive Anosov representations of a non-abelian free group are representations with an Anosov behaviour ``in restriction to'' primitive elements. More precisely, the second author proved in \cite{W23} that primitive Anosov representations are representations such that the associated local system over the geodesic flow of the free group admits a dominated splitting \emph{in restriction to} the closure of the union of all closed orbits corresponding to primitive elements. This motivates the more general study of \emph{restricted Anosov representations}, initiated in \cite{W23}, which we carry on in the present paper.

With this point of view, primitive Anosov representations have a natural analog for closed surface groups, which we call \emph{simple Anosov}:

\begin{definition}
    A representation $\rho$ of the fundamental group $\Gamma$ of a closed connected hyperbolic surface $S$ into $\SL(d,\Cb)$ is called \emph{simple $k$-Anosov} if the local system associated to $\rho$ over the geodesic flow of $S$ admits a dominated splitting of rank $k$ in restriction to the closure of the union of simple closed geodesics.
\end{definition}

We will call the closure of simple closed geodesics the \emph{Birman--Series set}, in reference to Birman and Series \cite{BS85}, who proved that this set is ``sparse'' in the unit tangent bundle of the hyperbolic surface (it has Hausdorff dimension $1$). We refer to Section \ref{RestrictedAnosovSection} for precisions about the above definition (in particular, the definition of dominated splitting of rank $k$).

We prove in Section \ref{MCGSection} the following expected property:
\begin{proposition} \label{p-MCGProperlyDiscontinuous}
    Let $\Gamma$ be the fundamental group of a closed connected hyperbolic surface $S$. Then the set of simple Anosov representations $\Gamma$ into $\SL(d,\Cb)$ modulo conjugation is an open domain of discontinuity for the action of $\Out(\Gamma)= \Mod(S)$.
\end{proposition}

For non-oriented surfaces, Lee proved in \cite{Lee11} that the domain of simple Anosov representations into $\SL(2,\Cb)$ contains stricly the domain of Anosov (i.e. quasifuchsian) representations. In contrast, for oriented surfaces, the domain of quasifuchsian representations is known to be a maximal domain of discontinuity by results of Lee \cite{Lee11} and Souto--Storm \cite{SS06}. The main goal of the present paper is to construct new examples of simple Anosov representations in higher rank:

\begin{theorem}
    Let $\Gamma$ be the fundamental group of a closed connected hyperbolic surface. Then, for every $d\geq 2$, there exists a simple $d$-Anosov representation of $\Gamma$ into $\SL(2d,\Cb)$ in the boundary of the domain of $d$-Anosov representations.
\end{theorem}

Taking a ``generic deformation'', we obtain the following:

\begin{corollary}
    Let $\Gamma$ be the fundamental group of a closed connected hyperbolic surface. Then, for every $d\geq 2$, there exist simple $d$-Anosov representation of $\Gamma$ into $\SL(2d,\Cb)$ with analytically dense image.
\end{corollary}

As a consequence, we obtain domains of discontinuity of the mapping class group action that properly contain the domain of Anosov representations.

\begin{corollary}
    Let $\Gamma$ be the fundamental group of a closed connected hyperbolic surface. Then, for every $d\geq 2$, there exists an open domain of discontinuity for the action of $\Out(\Gamma)$ on $\Hom(\Gamma, \SL(2d,\Cb))/\SL(2d,\Cb)$ which contains points of the boundary of domain of Anosov representations.
\end{corollary}

In particular, the domain of Anosov representations modulo conjugation is not a maximal domain of discontinuity.

\subsection{Further results and open questions}

\subsubsection{Simple $P$-Anosov representations into $G$}

More generally, there is a notion of simple $P$-Anosov representation into $G$, for any pair of a semisimple (real or complex) linear group $G$ and a parabolic subgroup $P$. Here we will focus on simple $P_d$-Anosov representations into $\SL(2d,\Cb)$, where $P_d$ is the stabilizer of a $d$-dimensional subspace of $\Cb^d$, but one could easily elaborate on these to build more examples. For instance, the direct sum of one our representations with a trivial representation will give simple $d$-Anosov representations into $\SL(d',\Cb)$ for any $d'\geq 2d$. Another example is the following: Among the exotic simple $2$-Anosov representations into $\SL(4,\Cb)$ that we construct, some take values in the complex symplectic group $\Sp(4,\Cb)$. Through the isomorphism $\Sp(4,\Cb) \simeq \SO(5,\Cb)$, one obtains simple $1$-Anosov representations into $\SO(5,\Cb)\subset \SL(5,\Cb)$.

These constructions, however, do not seem to exhaust all the possibilities. In particular, it seems that our construction cannot provide \emph{simple Borel Anosov} representations, i.e. simple Anosov with respect to a minimal parabolic subgroup. This raises the following question:

\begin{question}
    Does there exist a representation of a closed oriented surface group into $\SL(d,\Cb)$, $d\geq 2$ that is simple Borel Anosov but not Borel Anosov? In particular, does there exist a representation of a closed oriented surface group into $\SL(3,\Cb)$ which is simple Anosov but not Anosov ?
\end{question}

On a related topic, Maloni, Martone, Mazzoli and Zhang have initiated in \cite{MMMZ} the study of representations which are Borel Anosov in restriction to a fixed lamination. 

\subsubsection{Other mapping class group invariant subflows}

We develop more generally the basic properties of Anosov representations in restriction to a subflow of the geodesic flow. In particular, any subflow which is globally preserved by the mapping class group gives rise to a domain of discontinuity for the mapping class group on the character variety. We mention various examples in Section \ref{MCGSection}. There, we also prove that any closed subflow of the geodesic flow that is preserved by a finite index subgroup of the mapping class group contains the Birman--Series set. As a consequence, all these other domains of discontinuity associated to subflows are contained in the domain of simple Anosov representations. This provides evidence for a positive answer to the following question:

\begin{question}
    Is the domain of simple Anosov representations a maximal domain of discontinuity for the mapping class group action on a character variety of a surface group?
\end{question}

\subsection{Structure of the paper}
In Section \ref{PriliminariesSection}, we introduce a general notion of \emph{$\Gamma$-flow} of a finitely generated group, of which the main examples are subflows of geodesic flows of hyperbolic or relatively hyperbolic groups. In Section \ref{RestrictedAnosovSection}, we develop the general notion of Anosov representation \emph{in restriction to a $\Gamma$-flow}, which specially emphasis on relatively Anosov and simple Anosov representations. This section contains several general results of independent interest which will make the main construction rather natural.

Section \ref{ConstructionSection} presents our main construction of exotic simple Anosov representations. In a word, these are obtain as the induced representations of a geometrically finite but not quasifuchsian representation into $\SL(2,\Cb)$ of the fundamental group a covering of degree $d$. Finally, Section \ref{MCGSection} develops the applications to mapping class group actions on character varieties.

\subsection{Acknowledgements}
We thank Samuel Bronstein, Fr\'ed\'eric Paulin, Juan Souto, Tengren Zhang and Feng Zhu for references and interesting discussions related to our work.

\section{Geodesic flows of hyperbolic and relatively hyperbolic groups}\label{PriliminariesSection}

\subsection{Hyperbolic and relatively hyperbolic groups}
Let $(X,d_X)$ be a metric space. Recall that a map $\ell$ from $\Rb$ (resp. $\Rb_{\geq 0}$, $[a,b]\subset \Rb$) to $X$ is a \emph{geodesic} (resp. \emph{geodesic ray}, \emph{geodesic segment}) if \[d_X(\ell(t),\ell(s)) = \vert t-s\vert\] for all $s,t$ in the interval of definition. The space $X$ is called \emph{geodesic} if any two points are joined by a geodesic segment, and called \emph{proper} if closed balls are compact. We say that $X$ is \emph{taut} if every point is at uniformly bounded distance from a bi-infinite geodesic. When $(X,d_X)$ is proper and geodesic, we say that it is \emph{$\delta$-hyperbolic}, if for any geodesic triangle in $X$, there exists a point at distance at most $\delta$ from all three sides of the triangle, which is called a center of the triangle. The \emph{(Gromov) boundary} of a $\delta$-hyperbolic space $X$, denoted by $\partial_\infty X$, is defined to be the asymptotic classes of geodesic rays.

When $(X,d_X)$ is a $\delta$-hyperbolic space, a \emph{horofunction} about a boundary point $p\in \partial_\infty X$ is a function $h: X\to \Rb$ such that for any $x,x'\in X$,
\[|(h(x)-d_X(x,x_0))-(h(x')-d_X(x',x_0))|\] is uniformly bounded, where $x_0$ is a center of the ideal triangle with vertices $x, x'$ and $p$. A subset $B\subset X$ is called a \emph{horoball} centered at $p$ if there exists a horofunction $h$ about $p$, such that $B = \{x\in X| h(x)< 0\}$.

Finally, a finitely generated group $\Gamma$ is \emph{(Gromov) hyperbolic} if it admits a properly discontinuous and cocompact isometric action on a $\delta$-hyperbolic space $(X,d_X)$ (for some constant $\delta$). The space $X$ is called a \emph{Gromov model} of $\Gamma$.

The notion of Gromov hyperbolic groups is a far-reaching generalization of convex-cocompact Kleinian groups. Similarly, the notion of \emph{relatively hyperbolic groups} generalizes that of geometrically finite Kleinian groups. We follow here the definition given by Gromov \cite{Gromov}.

Let $\Gamma$ be a finitely generated group and let $\Pi = \{\Pi_i\}_{i\in I}$ be a finite collection of finitely generated, infinite subgroups of $\Gamma$. Let \[\Pi^{\Gamma} = \{\gamma\Pi_i\gamma^{-1}\ |\ \gamma\in \Gamma,\Pi_i\in\Pi\}\] be the collection of conjugates of the subgroups in $\Pi$.

\begin{definition}
The pair $(\Gamma,\Pi)$ is called a \emph{relatively hyperbolic pair}, and $\Gamma$ is called \emph{hyperbolic relative to} $\Pi$, if there exists a $\delta$-hyperbolic space $(X,d_X)$ with a properly discontinuous isometric action of $\Gamma$ such that
\begin{itemize}
    \item[(1).] $X$ is either taut or a horoball;
    \item[(2).] There exists $\mathcal{B}=\{B_i\}_{i\in I}$, a collection of horoballs of $X$, such that $\mathcal{B}^{\Gamma} = \{\gamma B_i\ |\ \gamma\in\Gamma, i\in I\}$ is a collection of disjoint open horoballs with $\gamma\Pi_i\gamma^{-1}$ the stabilizer of $\gamma B_i$ for each $\gamma\in\Gamma$ and $i\in I$;
    \item[(3).] $\Gamma$ acts on $X^{th} = X \setminus \bigcup_{i\in I,\gamma\in \Gamma} \gamma B_i $, the thick part of $X$, cocompactly.
\end{itemize}
The space $X$ is called a \emph{Gromov model of} $(\Gamma,\Pi)$. The Gromov boundary of $X$ is called the Bowditch boundary of $(\Gamma,\Pi)$, denoted by $\partial_\infty (\Gamma,\Pi) = \partial_\infty X$. The subgroups in $\Pi^{\Gamma}$ are called \emph{peripheral subgroups} of $\Gamma$ and the centers of horoballs in $\mathcal{B}^{\Gamma}$ are called \emph{parabolic points}.
\end{definition}

While two Gromov models $X$ and $X'$ of a relatively hyperbolic pair $(\Gamma,\Pi)$ are not necessarily quasi-isometric, there always exists a $\Gamma$-equivariant homeomorphism between their Gromov boundaries (see \cite{Bow12}). The Bowditch boundary $\partial_\infty (\Gamma,\Pi)$ is thus well-defined independently of choice of the Gromov models.

Recall the definition of a convergence group action.

\begin{definition}\label{defconvergencegroup}
    Let $Z$ be a metrizable compact set. An action of a discrete group $\Gamma$ by homeomorphisms on $Z$ is a \emph{convergence group action} if for every unbounded sequence $(\gamma_n)_{n\in \Nb}$ in $\Gamma$, there exists a subsequence $(\gamma_{n_k})_{k\in \Nb}$ and a pair of points $(x_-,x_+)\in Z^2$ such that 
    $\gamma_{n_k} z \to x_+$ as $k\to +\infty$
    for all $z\in K\setminus \{x_-\}$. The points $x_-$ and $x_+$ are called respectively the \emph{repelling} and \emph{attracting} points of the subsequence $(\gamma_{n_k})_{k\in \Nb}$. In this case, we say a point $p \in Z$ is
    \begin{itemize}
        \item[(a)] a \emph{conical limit point} if there exists a sequence $(\gamma_n)_{n\in \Nb}$ with attracting point $p$ and repelling point distinct from $p$;
        \item[(b)] a \emph{bounded parabolic point} if the stabilizer of $p$ in $\Gamma$ is infinite and acts properly discontinuously and cocompactly on $Z\setminus \{p\}$.  
    \end{itemize}
    We say that the $\Gamma$-action on $Z$ is \emph{geometrically finite} if $Z$ consists of only conical limit points and bounded parabolic points.
\end{definition}

If $\Gamma$ is a discrete group of isometries of a $\delta$-hyperbolic space $(X,d_X)$, then the action of $\Gamma$ on $\partial_\infty \Gamma$ is a convergence group action. Moreover, a point $p\in \partial_\infty X$ is a conical limit point if and only if there exists a sequence $(\gamma_n)_{n\in \Nb}$ such that, for some (hence any) point $o\in X$ and some (hence any) geodesic ray $\ell$ converging to $p$, the sequence $(\gamma_n o)_{n\in \Nb}$ converges to $p$ and remains at bounded distance from $\ell$.

When $(\Gamma,\Pi)$ is a relatively hyperbolic pair, $\Gamma$ acts on $\partial_\infty(\Gamma,\Pi)$ geometrically finitely (see \cite{Bow12}). Conversely, Yaman \cite{Yam} proved the following theorem.

\begin{theorem}[Yaman \cite{Yam}]
    Let $\Gamma$ be a finitely generated group with a convergence group action on a perfect metrizable compact space $Z$. If the $\Gamma$-action is geometrically finite, then there are finitely many orbits $\{\Gamma p_i\}_{i\in I}$ of bounded parabolic points, and $\Gamma$ is hyperbolic relative to $\Pi = \{\Stab_\Gamma(p_i)\}_{i\in I}$. Moreover, $Z$ is $\Gamma$-equivariantly homeomorphic to the Bowditch boundary of $(\Gamma,\Pi)$.
\end{theorem}

A relatively hyperbolic pair $(\Gamma,\Pi)$ is called \emph{elementary} if $\partial_\infty (\Gamma,\Pi)$ contains at most two points. This happens when the $\Pi_i$ are finite, and $\Gamma$ is finite or virtually isomorphic to $\Zb$, and when $\Pi$ consists of only one subgroup of $\Gamma$ of finite index. 

Otherwise, $\partial_\infty (\Gamma,\Pi)$ is a perfect metrizable compact set and the pair $(\Gamma, \Pi)$ is called \emph{non-elementary}. From now on, we always assume that the hyperbolic groups and hyperbolic pairs we consider are non-elementary.

We will be interested in the situation where a Gromov hyperbolic group also admits a relatively hyperbolic structure. This was studied by Osin \cite{Osin06}, Bowditch \cite{Bow12}, Tran \cite{Tran13}, Manning \cite{Man15} etc.

Let $\Gamma$ be a group, a finite collection $\Pi = \{\Pi_i\}_{i\in I}$ of subgroups of $\Gamma$ is called \emph{almost malnormal} if for any $i,j\in I$, $\Pi_i\cap \gamma\Pi_j\gamma^{-1}$ is infinite only when $i=j$ and $\gamma\in \Pi_i$.

\begin{theorem}[\cite{Osin06} Proposition 2.36, Lemma 5.4, \cite{Bow12} Theorem 7.11] \label{thm - Hyperbolic + Relatively Hyperbolic}
Let $\Gamma$ be a nonelementary hyperbolic group and $\Pi$ a finite, almost malnormal collection of quasiconvex subgroups. Then $(\Gamma,\Pi)$ is relatively hyperbolic. Conversely, if $\Gamma$ is both hyperbolic and hyperbolic relative to $\Pi$, then $\Pi$ is an almost malnormal collection of quasiconvex subgroups.
\end{theorem}

In the situation of Theorem \ref{thm - Hyperbolic + Relatively Hyperbolic}, the relation between the Gromov and Bowditch boundaries is described in the following way.

\begin{theorem}[\cite{Tran13} Main Theorem, \cite{Man15} Theorem 1.3]\label{relboundary}
    Let $\Gamma$ be a nonelementary hyperbolic group and let $\Pi$ be a finite almost malnormal collection of quasiconvex subgroups. Then there is a $\Gamma$-equivariant continuous surjective map \[\eta:\partial_\infty(\Gamma)\to \partial_\infty (\Gamma,\Pi)~,\] such that,
    \begin{itemize}
        \item[(a)] the preimage of a conical limit point by $\eta$ is a singleton;
        \item[(b)] the preimage of the bounded parabolic point $\gamma p_i$ fixed by $\gamma \Pi_i\gamma^{-1}$ is $\gamma\partial_\infty \Pi_i \subset \partial_\infty \Gamma$ for any $i\in I$ and $\gamma\in\Gamma$.
    \end{itemize}
\end{theorem}

Informally, this theorem states that the Bowditch boundary of $(\Gamma,\Pi)$ is obtained from the Gromov boundary of $\Gamma$ by contracting the Gromov boundary of each conjugate of some $\Pi_i, i\in I$ to a point.

\subsection{Geodesic flows}
We now introduce a very general notion of flow space for a finitely generated group $\Gamma$, of which an important source of examples will be given by geodesic flows of Gromov models of relatively hyperbolic groups. Let $\Gamma$ be a discrete group.

\begin{definition}
A $\Gamma$-\emph{flow} is the data of a Hausdorff topological space $Y$ with a continuous flow $\phi$ and a properly discontinuous action of $\Gamma$ on $Y$ that commutes with $\phi$. This $\Gamma$-flow is called \emph{cocompact} if the quotient $\Gamma \backslash Y$ is compact. A $\Gamma$-subflow of a $\Gamma$-flow $(Y,\phi)$ is a $\Gamma$-invariant subflow of $(Y,\phi)$.
\end{definition}

 \begin{remark}
     If $\Gamma'$ is a subgroup of $\Gamma$, then every $\Gamma$-flow is automatically a $\Gamma'$-flow by restriction of the action.
 \end{remark}

We now introduce a weak notion of \emph{morphism} between $\Gamma$-flows. Importantly, while such morphisms map orbits of one flow to orbits of the other, we do not require that they preserve the time of the flow.

\begin{definition}
A \emph{morphism} $\sigma$ between $\Gamma$-flows $(Y_1,\phi_1)$ and $(Y_2,\phi_2)$ is a $\Gamma$-equivariant continuous map $\sigma:Y_1\to Y_2$ for which there exists a map $c_\sigma: Y_1\times\Rb \to \Rb$ such that 
\[\sigma(\phi_1^t(y))= \phi_2^{c_\sigma(y,t)}(\sigma(x))\]
for all $(y,t)\in Y_1 \times \Rb$. Such a morphism $\sigma$ is \emph{quasi-isometric} if there exist constants $\lambda\geqslant 1$ and $\epsilon\geqslant 0$, such that for any $y\in Y_1$, the map $c_\sigma(y,\cdot):\Rb \to \Rb$ is $(\lambda,\epsilon)$-quasi-isometric and $c(y,t)\to +\infty$ as $t\to +\infty$. A morphism $\sigma$ is an isomorphism if it is a homeomorphism. In that case, $\sigma^{-1}$ is automatically a morphism from $(Y_2,\phi_2)$ to $(Y_1,\phi_1)$.
\end{definition}

\begin{remark}
    The map $c_\sigma: Y_1\times\Rb \to \Rb$ in the above definition is $\Gamma$-invariant and satisfies the cocycle rule
    \[c_\sigma(y,t+s)=c_\sigma(y,t)+ c_\sigma(\phi_1^t(y),s)\]
    for all $y\in Y_1$ and all $s,t\in \Rb$.
\end{remark}

If $\Gamma$ acts properly discontinuously by isometries on a $\delta$-hyperbolic space $(X,d_X)$, a direct way to construct a $\Gamma$-flow is to consider the collection of parametrized geodesics of $X$. Define \[\mathcal{G}(X)=\{\ell\ |\ \ell:\Rb \to X\text{ geodesic}\}~,\] equipped with the flow $\psi$ defined by \[\psi^t(\ell): s \mapsto \ell(s + t)~.\] Then $(\mathcal G(X), \psi)$ is a $\Gamma$-flow. One can define the metric $d'$ on $\mathcal{G}(X)$ by \[d'(\ell,\ell') = \int_{\Rb} \frac{e^{-\abs{t}}}{2}d_X\big(\ell(s),\ell'(s)\big)ds~,\] for any $\ell,\ell'\in \mathcal{G}(X)$ and verify that the projection $\pi':\mathcal{G}(X)\to X$ defined by $\pi'(\ell)=\ell(0)$ is a $\Isom(X)$-equivariant quasi-isometry.\\

The drawback of this construction is that quasi-isometries between hyperbolic spaces do not induce morphisms of flows. For our purposes, we will need a more ``canonical'' notion of geodesic flow, typically one for which pairs of distinct points in the boundary at infinity define a unique orbit of the flow. Such a flow is provided to us by a general theorem of Mineyev. We denote by $A^{(2)}$ the set of ordered pairs of distinct points in a set $A$.

\begin{theorem}[\cite{Mineyev} Theorem 60]\label{mineyevflow}
Let $(X,d_X)$ be a taut hyperbolic metric space and $\Gamma$ a discrete subgroup of $\Isom(X)$. Then there exists a metric $d$ on the topological space $\mathcal{F}(X)=\partial_{\infty}^{(2)}X \times \Rb$ and a continuous cocycle $c:\Gamma \times \partial_\infty X^{(2)} \to \Rb$ with the following properties:
\begin{itemize}
    \item[(1).] The action of $\Gamma$ on $\mathcal F(X)$ given by
    \[\gamma (z_-,z_+,t) = (\gamma  z_-, \gamma z_+, t+ c(z_-,z_+,\gamma))\]
    is properly discontinuous and isometric;
    
    \item[(2).] There is a $\Gamma$-equivariant quasi-isometry $\pi:(\mathcal{F}(X),d)\to (X,d_X)$;

    \item[(3).] There exist constants $\lambda\geqslant 1$ and $\epsilon\geqslant 0$ such that for any $(z_-,z_+) \in \partial_\infty X^{(2)}$, the curve $\{\pi(z_-,z_+,t), t\in \Rb\}$ is a $(\lambda,\epsilon)$-quasi-geodesic in $X$ with backward endpoint $z_-$ and forward endpoint $z_+$.
\end{itemize}
\end{theorem}

Note that the definition of the action of $\Gamma$ on $\mathcal F(X)$ implies that it commutes with the flow $\phi$ defined by \[\phi^t(z_-,z_+,s) = (z_-,z_+,t+s)~.\] Hence $(\mathcal F(X),\phi)$ is a $\Gamma$-flow.

When $\Gamma$ is a hyperbolic group and $X$ is a Gromov model of $\Gamma$ (e.g. its Cayley graph with respect to a finite generating set), we define \emph{geodesic flow of $\Gamma$} to be the space $\mathcal F(\Gamma) \equaldef \mathcal F(X)$ equipped with the flow $\phi$ and the action of $\Gamma$. This $\Gamma$-flow is well-defined up to quasi-isometric isomorphisms.

When $X$ is a Gromov model of a relatively hyperbolic pair $(\Gamma,\Pi)$ (e.g. its \emph{Groves--Manning cusp space} as defined in \cite{GM08}),  we call \emph{geodesic flow of $(\Gamma,\Pi)$} the space $\mathcal{F}(\Gamma,\Pi) \equaldef \mathcal{F}(X)$ equipped with the flow $\phi$ and the action of $\Gamma$. Now, it is only well-defined up to isomorphisms of $\Gamma$-flow, since different Gromov models of $(\Gamma,\Pi)$ may not be quasi-isometric to each other. The next proposition is independent to the choice of the Gromov models for defining $\mathcal{F}(\Gamma,\Pi)$.

\begin{proposition}\label{relofsubflows}
Let $\Gamma$ be a hyperbolic group and $\Pi$ a finite almost malnormal collection of quasi-convex subgroups. Then there is a morphism of $\Gamma$-flows
\[\sigma: \mathcal{F}(\Gamma) \setminus\bigcup_{i \in I,\gamma \in \Gamma} \gamma \mathcal  F(\Pi_i) \longrightarrow \mathcal{F}(\Gamma,\Pi)\] that extends the quotient map $\eta:\partial_\infty\Gamma \to \partial_\infty(\Gamma,\Pi)$ from Theorem \ref{relboundary}. 
Moreover, the restriction of $\sigma$ to any cocompact $\Gamma$-subflow is a quasi-isometric morphism.
\end{proposition}
\begin{proof}
    Let $X$ be any Gromov model of the relatively hyperbolic pair $(\Gamma,\Pi)$. We fix $\mathcal{F}(\Gamma,\Pi)$ to be the $\Gamma$-flow $\mathcal{F}(X)$.
    Consider the space 
    \[U = \Big\{\Big((z_-,z_+,s_1),(\eta z_-,\eta z_+,s_2)\Big): z_-,z_+\in \partial_\infty \Gamma, s_1, s_2 \in \Rb, \eta z_- \neq \eta z_+ \Big \} \subset \mathcal{F}(\Gamma)\times \mathcal{F}(\Gamma,\Pi)\]
    There is a natural $\Gamma$-action on $U$ by coordinates since $\eta$ is $\Gamma$-equivariant. Note that the projection to the first coordinate maps $U$ to $\mathcal{F}(\Gamma) \setminus\bigcup_{i \in I,\gamma \in \Gamma} \gamma \mathcal{F}(\Pi_i)$ by Theorem \ref{relboundary}. This projection factors to a fiber bundle  \[p_1: \Gamma\backslash U \to \Gamma\backslash\Big(\mathcal{F}(\Gamma) \setminus\bigcup_{i \in I,\gamma \in \Gamma} \gamma \mathcal{F}(\Pi_i)\Big) \] with fibers homeomorphic to $\Rb$.
    Let $\hat{\sigma}$ be a continuous section of $p_1$. Lifting $\hat{\sigma}$ to $\mathcal{F}(\Gamma) \setminus\bigcup_{i \in I,\gamma \in \Gamma} \gamma \mathcal{F}(\Pi_i)$ and composing with the projection to the second coordinate, we obtain a morphism of $\Gamma$-flows
    \[\sigma: \mathcal{F}(\Gamma) \setminus\bigcup_{i \in I,\gamma \in \Gamma} \gamma \mathcal{F}(\Pi_i) \longrightarrow \mathcal{F}(X(\Gamma,\Pi))\] 
    which extends the boundary map $\eta$.
    
    It remains to prove that $\sigma$ is quasi-isometric in restriction to any cocompact $\Gamma$-subflow. Let us thus introduce the cocyle
    \[c: \Big(\mathcal{F}(\Gamma) \setminus\bigcup_{i \in I,\gamma \in \Gamma} \gamma \mathcal{F}(\Pi_i)\Big)\times \Rb\longrightarrow \Rb\] such that \[\sigma(\phi^t(z))=\psi^{c(z,t)}(\sigma(z))\] for any $z\in \mathcal{F}(\Gamma) \setminus\bigcup_{i \in I,\gamma \in \Gamma} \gamma \mathcal{F}(\Pi_i) $ and $t\in\Rb$.
    
    Let $\widetilde{K}$ be a cocompact $\Gamma$-subflow of $\mathcal{F}(\Gamma) \setminus\bigcup_{i \in I,\gamma \in \Gamma} \gamma \mathcal{F}(\Pi_i)$. Since $c$ satisfies the cocycle rule
    \[c(x,t+s) = c(x,t) + c(\phi^t(x),s)~,\]
    in order to prove that it is quasi-isometric, it suffices to verify the following statement.
    
    \begin{claim}
        There exist constants $T,m,M>0$, such that $m\leqslant c(z,T) \leqslant M$ for any $z\in\widetilde{K}$.
    \end{claim}
    
    Let $K$ be a compact subset of $\mathcal{F}(\Gamma) \setminus\bigcup_{i \in I,\gamma \in \Gamma} \gamma \mathcal{F}(\Pi_i)$ such that $\Gamma K = \widetilde{K}$. Since $c$ is continuous and $\Gamma$-invariant, we only need to show that there exists a constant $T>0$ such that $c(z,T)>0$ for any $z\in K$.
    
    We argue by contradiction. Suppose that there is a sequence $(z^n)_{n\in \Nb}\subset K$ and a sequence $(T_n)_{n\in \Nb}\subset \Rb_+$ with $T_n\to +\infty$ as $n\to +\infty$, such that $c(z^n,T_n)\leqslant 0$.

    Up to a subsequence, we may assume $z^n=(z^n_-,z^n_+,s_n)$ converges to $z=(z_-,z_+,s) \in K$ as $n\to +\infty$. For each $n$, there exists $\gamma_n\in \Gamma$ such that $\gamma_n^{-1}\phi^{T_n}(z^n)\in K$ by $\Gamma K=\widetilde{K}$. Then for a base point $z^0\in K$, we have $\gamma_n z^0 \to z_+$ as $n\to +\infty$ since $d(\gamma_n z^0,\phi^{T_n}(z^n)) = d( z^0,\gamma_n^{-1}\phi^{T_n}(z^n))$ is bounded by the diameter of $K$ and $\phi^{T_n}(z^n)\to z_+$ as $n\to +\infty$. Therefore, the sequence $(\gamma_n)_{n\in \Nb}$ has attracting point $z_+$ for the convergence group action of $\Gamma$ on $\partial_\infty\Gamma$.

    We apply the argument above by replacing $z^n$ by $\sigma(z^n)$ in the compact set $\sigma(K)$, and replacing $T_n$ by $c(z^n,T_n)$. Since $c(z^n,T_n)\leqslant 0$, after further extraction, either $\gamma_n \sigma(z^0)$ is bounded in $\mathcal{F}(X(\Gamma,\Pi))$ or $\gamma_n \sigma(z^0)\to \eta(z_-)$ as $n\to +\infty$, depending on whether $c(z^n,T_n)$ is bounded or diverges to $-\infty$. Since $\gamma_n$ is unbounded and $\Gamma$ acts properly on $\mathcal F(X)$, we conclude that $\gamma_n \sigma(z^0)\to \eta(z_-)$ as $n\to +\infty$ and $\eta(z_-)$ is the attracting point of $(\gamma_n)$ in $\partial_\infty(\Gamma,\Pi)$.

    On the other hand, since $\eta$ is $\Gamma$-equivariant, it must send the attracting point of $\gamma_n$ in $\partial_\infty \Gamma$ to its attracting point in $\partial_\infty (\Gamma, \Pi)$. We deduce that $\eta(z_+) = \eta(z_-)$. By Proposition \ref{relboundary}, this implies that $z_-$ and $z_+$ both belong to $\gamma \partial_\infty P_i$ for some $i\in I$ and some $\gamma \in \Gamma$, contradicting the assumption that \[z\in K \subset \mathcal{F}(\Gamma) \setminus\bigcup_{i \in I,\gamma \in \Gamma} \gamma \mathcal{F}(\Pi_i)~.\]
\end{proof}

\section{Restricted Anosov and relatively Anosov representations}\label{RestrictedAnosovSection}

We always fix the field $\Kb$ to be $\Rb$ or $\Cb$.

The notion of \emph{Anosov representation} of a hyperbolic group $\Gamma$ admits many equivalent definitions. One of them, which is close to Labourie's original definition and was developed by Bochi--Potrie--Sambarino \cite{BPS}, states that a linear representation $\rho:\Gamma \to \SL(d,\Kb)$ is $k$-Anosov if the associated flat bundle admits a dominated splitting over the geodesic flow of $\Gamma$. The purpose of this section is to investigate the generalization of this definition when the geodesic flow is replaced by any $\Gamma$-flow.

\subsection{Linear representations and dominated splittings}
Let $\Gamma$ be a countable group and let $(Y,\phi)$ be a $\Gamma$-flow. For a representation $\rho:\Gamma\to \SL(d,\Kb)$, we consider the flat bundle $E_\rho(Y) = \Gamma\backslash (Y\times \Kb^d)$, where the $\Gamma$-action is given by $\gamma(y,v)=(\gamma y, \rho(\gamma)v)$ for all $\gamma\in\Gamma$, $y\in Y$ and $v\in \Kb^d$. The flow $\phi$ on $Y$ lifts to a flow on $Y\times \Kb^d$ by parallel transformations, which we still denote by $\phi$, namely, \[\phi^t(y,v) = (\phi^t(y),\phi^t_y(v)) = (\phi^t(y), v)\] for all $y\in Y$, $v\in\Kb^d$ and $t\in\Rb$. This flow commutes with the $\Gamma$-action and thus factors to a flow on $E_\rho(Y)$, which we again denote by $\phi$.

\begin{definition}\label{restAnosov}
    A representation $\rho$ of $\Gamma$ into $\SL(d,\Kb)$ is \emph{$k$-Anosov in restriction to the $\Gamma$-flow $(Y,\phi)$} if there exists a metric $\Vert\cdot\Vert$ on the vector bundle $E_\rho(Y)$ such that $E_{\rho}(Y)$ admits a \emph{dominated splitting of rank $k$}, that is, a continuous $\phi$-invariant decomposition $E_\rho(Y) = E_{\rho}^s \oplus E_{\rho}^u$ with $E_{\rho}^s$ of rank $k$, for which there exist constants $C,\lambda >0$, such that \[ \dfrac{\Vert \phi^t_y(v)\Vert}{\Vert \phi^t_y(w)\Vert} \leqslant C e^{-\lambda t}\dfrac{\Vert v\Vert}{\Vert w\Vert}~,\] for all $y\in \Gamma\backslash Y$, $t\in \Rb_+$ and all non zero vectors $v \in (E^s_\rho)_y$and $w\in (E^u_\rho)_y$. We respectively call $E^s_\rho$ and $E^u_\rho$ the \emph{stable direction} and \emph{unstable direction} of $E_\rho(Y)$ with respect to the metric $\Vert \cdot \Vert$.
\end{definition}

\begin{remark}\label{cptnormind}
 When the $\Gamma$-flow $(Y,\phi)$ is cocompact, the dominated splitting is unique and does not depend on the choice of the metric, since any two metrics on $E_\rho(Y)$ are uniformly bi-Lipschitz.
\end{remark}

\begin{remark}\label{abuseliftnotations}
    By abuse of notations, we will write $Y\times \Kb^d = E^s_\rho \oplus E^u_\rho$ to represent the lift of the dominated splitting over $Y$, as a decomposition of $Y\times \Kb^d$ into $\Gamma$-invariant, $\phi$-invariant subbundles. The metric $\Vert \cdot \Vert$ will lift to a $\Gamma$-invariant one, still denoted by $\Vert \cdot \Vert$. We call it a dominated splitting of $Y\times \Kb^d$ of rank $k$ associated to $\rho$ and $\Vert\cdot\Vert$.
\end{remark}

\begin{remark}\label{rescalingnorm}
    Let $\Vert\cdot\Vert_0$ denote the standard metric on $\Kb^d$. Since $Y\times \Kb^d$ is a trivial bundle, for any metric $\Vert \cdot \Vert$, there exists a continuous map $A: Y\to \GL(d,\Kb)$, such that at any point $y\in Y$, the metric $\Vert \cdot \Vert$ is expressed by $\Vert A(y)\cdot\Vert_0$. We will say that $\Vert \cdot \Vert$ is of unit volume if there exists such a map $A$ takes values in $\SL(d,\Kb)$. Since a rescaling of the metric preserves the ratio of the norms of two vectors, we may assume without loss of generality that the metric in Definition \ref{restAnosov} is always of unit volume.
\end{remark}

One of the good properties of the restricted Anosov definition is that it is preserved under pull-back by quasi-isometric morphisms of $\Gamma$-flows.

\begin{proposition}\label{anosovthroughmorphism}
    Let $\sigma: (Y_2, \phi_2) \to (Y_1, \phi_1)$ be a quasi-isometric morphism of $\Gamma$-flows, and let $\rho: \Gamma\to \SL(d,\Kb)$ be a $k$-Anosov representation in restriction to $(Y_1,\phi_1)$. Then $\rho$ is $k$-Anosov in restriction to $(Y_2,\phi_2)$.
\end{proposition}

\begin{proof}
Let $c: Y\times\Rb \to \Rb$ be the cocycle such that  \[\sigma(\phi_2^t(y)) = \phi_1^{c(y,t)}(\sigma(y))~.\] By definition of a quasi-isometric morphism, there exist constants $a>0,b>0$ such that \[c(y,t) \geq a t -b\] for all $y\in Y$ and all $t\geq 0$.

The morphism $\sigma$ naturally lifts to a continuous bundle morphism $\sigma: E_\rho(Y_1) \to E_\rho(Y_2)$. Let $\Vert \cdot \Vert$ be a continuous metric for which $E_\rho(Y_1)$ has a dominated splitting of rank $k$ \[E_\rho(Y_1) = E_{\rho}^s(Y_1) \oplus E_{\rho}^u(Y_1)~.\]

Pulling back this splitting by $\sigma$, we get a continuous $\phi_1$-invariant splitting
\[E_\rho(Y_2) = E_{\rho}^s(Y_2) \oplus E_{\rho}^u(Y_2)~.\]
Let us still denote by $\Vert \cdot \Vert$ the pull-back by $\sigma$ of the metric on $E_\rho(Y_1)$. With these choices, we have for all $y\in \Gamma \setminus Y_2$ and all $v\in E_{\rho}^s(Y_2)$, $w\in E_{\rho}^u(Y_2)$:
\begin{eqnarray*}
\frac{\Vert \phi_2^t(v)\Vert}{\Vert \phi_2^t(w)\Vert}&=& \frac{\Vert \sigma(\phi_2^t(v))\Vert}{\Vert \sigma(\phi_2^t(w))\Vert} = \frac{\Vert \phi_1^{c(y,t)}(\sigma (v))\Vert}{\Vert \phi_1^{c(y,t)}(\sigma(w))\Vert}\\
&\leqslant & C e^{-\lambda c(y,t)} \frac{\Vert \sigma(v)\Vert}{\Vert \sigma(w)\Vert} \\ &\leqslant & C e^b e^{-a\lambda t} \frac{\Vert v\Vert}{\Vert w\Vert}~,
\end{eqnarray*}
showing that the splitting $E_{\rho}^s(Y_2) \oplus E_{\rho}^u(Y_2)$ is dominated.
\end{proof}

Another good property of the notion of restricted Anosov representation is its stability under passing to a subgroup. Let $\Gamma$ be a countable group, $(Y,\phi)$ a $\Gamma$-flow and $\Gamma'$ a subgroup of $\Gamma$. Then $(Y,\phi)$ can be seen as a $\Gamma'$-flow by resctricting the $\Gamma$-action to $\Gamma'$.

\begin{proposition}
\label{subgptogp}
    Let $\rho: \Gamma \to \SL(n,\Kb)$ be a linear representation. 
    \begin{itemize}
    \item[(1)]If $\rho$ is $k$-Anosov in restriction to $(Y,\phi)$. Then $\rho|_{\Gamma'}$ is a $k$-Anosov in restriction to $(Y,\phi)$.
    \item[(2)] Conversely, assume that $(Y,\phi)$ is cocompact and $\Gamma'$ has finite index in $\Gamma$. If $\rho|_{\Gamma'}$ is $k$-Anosov in restriction to $(Y,\phi)$, then $\rho$ is $k$-Anosov in restriction to $(Y,\phi)$. 
    \end{itemize}
\end{proposition}

\begin{proof}
For part (1), let $\Vert \cdot \Vert$ be a $\Gamma$-invariant norm on $E_\rho(Y)$ and let $E_\rho(Y)= E_\rho^s(Y)\oplus E_\rho^u(Y)$ be a $\Gamma$-invariant and $\phi$-invariant splitting of rank $k$ that is dominated for $\Vert \cdot \Vert$. Then the splitting and the norm are in particular $\Gamma'$-invariant and define a dominated splitting over $\Gamma'\setminus Y$.

Now we prove part (2). By $(1)$ we can restrict to a smaller subgroup of finite index and assume that $\Gamma'$ is normal in $\Gamma$. let $E_\rho(Y) = E_\rho^s(Y)\oplus E_\rho^u(Y)$ be a $\phi$-invariant and $\Gamma'$-invariant splitting of rank $k$ over $Y$, which is dominated for a $\Gamma'$-invariant norm $\Vert \cdot \Vert$. Since $\Gamma'$ acts cocompactly on $Y$, the domination condition does not depend on the choice of the norm and we can assume without loss of generality that $\Vert \cdot \Vert$ is in fact $\Gamma$-invariant.
    
For $\gamma \in \Gamma$, consider the push-forward of the splitting, defined by
\[\gamma_* E_\rho^s(Y)_{x} \oplus \gamma_* E_\rho^u(Y)_x = \rho(\gamma) E_\rho^s(Y)_{\gamma^{-1}x} \oplus \rho(\gamma) E_\rho^u(Y)_{\gamma^{-1}x} \]
Since the action of $\Gamma$ on $Y\times \Rb^d$ commutes with the flow $\phi$, the push-forward splitting is again $\phi$-invariant. Moreover, for $\eta\in \Gamma'$, we have
\begin{eqnarray*}
\gamma_* E_\rho^s(Y)_{\eta x} & = & \rho(\gamma) E_\rho^s(Y)_{\gamma^{-1} \eta x} \\
&=& \rho(\gamma) E_\rho^s(Y)_{\gamma^{-1} \eta \gamma (\gamma^{-1}x)}\\
&=& \rho(\gamma) \rho(\gamma^{-1}\eta \gamma) E_\rho^s(Y)_{\gamma^{-1} x} \\
& & \textrm{since $\Gamma'$ is normal in $\Gamma$ and $E_\rho^s(Y)$ is $\Gamma'$ invariant}\\
&=& \rho(\eta) \gamma_* E_\rho^s(Y)_{x}~.
\end{eqnarray*}
The same holds for $\gamma_*E_\rho^u(Y)$, showing that the push-forward splitting is again $\Gamma'$ invariant. Finally, since $\Vert \cdot \Vert$ is $\Gamma$-invariant, we get that the push-forward splitting is a dominated splitting over $\Gamma'\setminus Y$. By uniqueness of such a splitting we conclude that
\[\gamma_* E_\rho^s(Y) \oplus \gamma_* E_\rho^u(Y) = E_\rho^s(Y) \oplus E_\rho^u(Y)~. \]
Hence $E_\rho^s(Y) \oplus E_\rho^u(Y)$ is a dominated splitting over $\Gamma \setminus Y$ and $\rho$ is $k$-Anosov in restriction to $Y$.
\end{proof}

Finally, one of the main properties of restricted Anosov representations over cocompact $\Gamma$-flows is their stability under small deformation.

\begin{proposition}\label{stabilityofRA}
    Let $(Y,\phi)$ be a cocompact $\Gamma$-flow. Then the space \[ A^k_Y(\Gamma, \SL(d,\Kb)) = \{\rho: \Gamma \to \SL(d,\Kb)\text{ $k$-Anosov in restriction to $Y$}\}\]
    is open in $\Hom(\Gamma,\SL(d,\Kb))$.
\end{proposition}

Proposition \ref{stabilityofRA} essentially follows from a general stability theorem for dominated splittings (see for instance \cite{Shub} Corollary 5.19. \cite{CZZ} Theorem 8.1 or \cite{W23} Theorem 7.1). A key point is that one can see the linear flows associated to representations in the neighbourhood of a representation $\rho_0$ as perturbations of the flow $\phi$ on the fixed vector bundle $E_\rho(Y)$. This is ensured by the following Lemma.

\begin{lemma}\label{isolem}
     Let $\rho_0:\Gamma\to \SL(d,\Kb)$ be a representation. Then there exists a neighborhood $O$ of $\rho_0$ in $\Hom(\Gamma,\SL(d,\Kb))$ and a continuous map
     \[g: O\times Y \to \GL(d,\Kb), \] such that \[\rho(\gamma)g(\rho,y)=g(\rho,\gamma y)\rho_0(\gamma)\]
     and $g|_{\rho_0\times Y}\equiv Id$.
\end{lemma}
\begin{proof}
    Let $K$ be a compact subset of $Y$ such that $\Gamma K=Y$. Let $U$ be an open, relatively compact subset of $Y$ that contains $K$. Then there exists a continuous function $f:Y\to \Rb_{\geqslant 0}$ such that $f=1$ on $K$ and $f=0$ on $Y\setminus U$. Define 
    \[\begin{array}{cccc}g: & \Hom(\Gamma,\SL(d,\Kb))\times Y &\to & Mat_{d\times d}(\Kb)\\
    & (\rho,y)&\mapsto & \dfrac{1}{\sum_{\gamma\in\Gamma} f(\gamma^{-1}y)} \sum_{\gamma\in\Gamma} f(\gamma^{-1}y)\rho(\gamma)\circ \rho_0(\gamma)^{-1}~.
    \end{array}\]
    Note that $f(\gamma^{-1}y)= 0$ for all but finitely many $\gamma$ (by the properness of the $\Gamma$ action and relative compactness of $U$), $\sum_{\gamma\in\Gamma} f(\gamma^{-1}y) \geq 1$ since $\Gamma y \cap K \neq \emptyset$ and $f \equiv 1$ on $K$. Hence $g$ is well-defined and continuous. One easily verifies that
    \begin{equation}\label{eq: equivariance g(rho,gamma)}
    \rho(\gamma) g(\rho, y) = g(\rho, \gamma y) \rho_0(y)\end{equation} 
and that $g|_{\rho_0 \times Y} \equiv \Id$. By the continuity of $g$, there is a neighborhood $O$ of $\rho_0$ such that $g|_{O\times K}$ takes values in $\GL(d,\Kb)$. Finally the equivariance property (\ref{eq: equivariance g(rho,gamma)}) implies that $g|_{O\times Y}$ takes values in $\GL(d,\Kb)$.
\end{proof}

The equivariance property of $g$ precisely means that $g(\rho,\cdot)$ factors to a bundle isomorphism from $E_{\rho_0}(Y)$ to $E_\rho(Y)$ which depends continuously on $\rho$. The rest of the proof of Proposition \ref{stabilityofRA} follows classical stability arguments for dominated splittings. We sketch here for completeness.

\begin{proof}[Proof of Proposition \ref{stabilityofRA}]
Let $\iota_0$ denote the $\Gamma$-action on $O\times Y\times \Kb^d$ defined by $\gamma(\rho,y,v)=(\rho,\gamma y,\rho_0(\gamma)y)$ for any $\rho\in O$, $\gamma\in\Gamma$ and $v\in\Kb^d$, and let $\iota$ denote the $\Gamma$-action on $O\times Y\times \Kb^d$ defined by $\gamma(\rho,y)=(\rho,y,\rho(\gamma)v)$ $\rho\in O$, $\gamma\in\Gamma$ and $v\in\Kb^d$. It is natural to extend $\phi$ on $O\times Y$ by $\phi^t(\rho,y)=(\rho,\phi^t(y))$.\\

The map $g$ given by the lemma above induces a $\iota_0$-$\iota$-equivariant map \[ O\times Y \times \Kb^d \simeq O\times Y \times \Kb^d\]
by $(\rho,y,v)\to (\rho,y,g(\rho,y)v)$ for any $\rho\in O$, $y\in Y$ and $v\in\Kb^d$, which is an isomorphism of vector bundles fibring over $id_{O\times Y}$, and hence induces a continuous isomorphism \[\hat{g}: \iota_0(\Gamma)\backslash(O\times Y \times \Kb^d) \simeq \iota(\Gamma)\backslash(O\times Y \times \Kb^d)~.\]

A dominated splitting $E_{\rho_0}(Y)= E^s_{\rho_0}\oplus E^u_{\rho_0}$ induces a dominated splitting \[\iota_0(\Gamma)\backslash (O\times Y\times \Kb^d) = (O\times E^s_{\rho_0})\oplus (O\times E^u_{\rho_0})\] with respect to a $\Gamma$-invariant metric. Then there is a decomposition \[\iota(\Gamma)\backslash (O\times Y\times \Kb^d) = \hat{g}(O\times E^s_{\rho_0}) \oplus \hat{g}(O\times E^u_{\rho_0})~.\] We denote $E_O^s=\hat{g}(O\times E^s_{\rho_0})$ and $E_O^u=\hat{g}(O\times E^u_{\rho_0})$ in brief. The decomposition is $\phi$-invariant over $\{\rho_0\}\times \Gamma\backslash Y$, but may not be $\phi$-invariant over the whole $O\times \Gamma\backslash Y$. We wish to make it $\phi$-invariant by small deformation. More concretely, we define a flow $\Phi$ (respectively, $\Psi$) on the space of continuous sections of $\Hom({E}_O^u, {E}_O^s)$ (respectively, $\Hom({E}_O^s, {E}_O^u)$) with norm at most $1$, such that $\phi^t$ maps the graph of $f_{(\rho,y)}$ to the graph of $\Phi^t(f_{(\rho,y)})$ (respectively, $\Psi^t(f_{(\rho,y)})$)for any section $f$ with norm at most $1$ and $(\rho,y)\in O\times \Gamma\backslash Y$. Following the argument of Lemma 7.4 in \cite{W23}, $\Phi$ and $\Psi$ are well-defined contraction maps. The images of the unique fixed point of $\Phi$ and the unique fixed point of $\Psi$, which are independent to $t$, give a new decomposition of $\iota(\Gamma)\backslash (O\times Y\times \Kb^d)$, which is $\phi$-invariant. Up to replacing $O$ by a smaller neighborhood of $\rho_0$, this decomposition of $\iota(\Gamma)\backslash (O\times Y\times \Kb^d)$ gives a dominated splitting for each piece $\iota(\Gamma)\backslash (\{\rho\} \times Y\times \Kb^d)\simeq E_{\rho}(Y)$, which completes the proof.
\end{proof}

\subsection{Relatively Anosov representations}
While Anosov representations are meant to generalize convex-cocompact representations to higher rank Lie groups, the notion of \emph{relatively Anosov representation} introduced by Zhu \cite{Zhu} and Zhu--Zimmer \cite{ZZ1}, which is equivalent to that of \emph{asymptotically embedded representation} introduced previously by Kapovich--Leeb \cite{KL}, is meant to extend to higher rank the geometrically finite representations of relatively hyperbolic groups.

Let $(\Gamma,\Pi)$ be a relatively hyperbolic pair.

\begin{definition}
 A representation $\rho:\Gamma\to \SL(d,\Kb)$ is $k$-Anosov relative to $\Pi$ if there exists a pair of continuous maps
\[\xi = (\xi^k,\xi^{d-k}):\partial_\infty (\Gamma, \Pi) \to \Gr_K(\Kb^d) \times Gr_{d-k}(\Kb^d)\]
which is 
\begin{itemize}
    \item $\rho$-equivariant, that is, $\rho(\gamma)\xi (\cdot) = \xi(\gamma( \cdot))$ for any $\gamma\in\Gamma$;
    \item transverse, that is, $\xi^k(x)\oplus \xi^{d-k}(y) = \Kb^d$ for any $x\ne y \in \partial_\infty (\Gamma,\Pi)$;
    \item strongly dynamics preserving, that is, for any sequence $(\gamma_n)_{n\in \Nb}\subset \Gamma$ such that $\gamma_n\to x\in \partial_\infty (\Gamma,\Pi)$ and $\gamma_n^{-1}\to y\in \partial_\infty (\Gamma,\Pi)$ as $n\to +\infty$, one has $\rho(\gamma_n)V \to \xi^k(x)$ as $n\to +\infty$ for any $V\in \Gr_k(\Kb^d)$ that transverse to $\xi^{d-k}(y)$.
\end{itemize}
\end{definition}
Let $X$ be a Gromov model of $(\Gamma,\Pi)$. Let $\rho:\Gamma \to \SL(d,\Kb)$ be a $k$-Anosov representation relative to $\Pi$. 
The pair of transverse boundary maps associated to $\rho$ defines a $\Gamma$-invariant and $\phi$-invariant splitting of $E_\rho(\mathcal G(X))$, which Zhu and Zimmer prove to be dominated:

\begin{theorem}[\cite{ZZ1} Theorem 1.3]\label{relativeimplyflowanosov}
Let $\rho:\Gamma \to \SL(d,\Kb)$ be a $k$-Anosov representation relative to $\Pi$. Then $\rho$ is $k$-Anosov in restriction to the $\Gamma$-flow $\mathcal G(X)$.
\end{theorem}

\begin{remark}
The notion of being ``Anosov in restriction $\mathcal G(X)$'' is called ``Anosov relative to $X$'' in \cite{ZZ1}.
\end{remark}

The main example of a relatively representation is the inclusion of a \emph{geometrically finite subgroup} of $\SL(2,\Cb)$. Recall that a subgroup $\Gamma$ of $\SL(2,\Cb)$ is called a \emph{geometrically finite subgroup} if the action of $\Gamma$ on its limit set $\Lambda(\Gamma) \subset \Cb \mathbf P^1$ is geometrically finite in the sense of Definition \ref{defconvergencegroup} (see \cite{Bow12}). In particular, such a group $\Gamma$ is hyperbolic relatively to the collection $\Pi$ of stabilizers of its parabolic points.

\begin{proposition}[\cite{ZZ2} Proposition 1.7]\label{gfimpliesra}
Let $\Gamma$ be a geometrically finite subgroup of $\SL(2,\Cb)$ and $\Pi$ the collection of its parabolic stabilizers. Then the inclusion $\Gamma \hookrightarrow \SL(2,\Cb)$ is $1$-Anosov relative to $\Pi$.
\end{proposition}

When $\Gamma \subset \SL(2,\Cb)$ is geometrically finite with $\Pi$ a set of representatives of the conjugacy classes of maximal parabolic subgroups of $\Gamma$, the convex hull of $\Lambda(\Gamma)$ in $\Hb^3$, denoted by $\mathcal C(\Gamma)$, is a Gromov model of the relatively hyperbolic pair $(\Gamma,\Pi)$. Since $\mathbb{H}^3$ is uniquely geodesic, $\mathcal C(\Gamma)$ contains a unique bi-infinite geodesic between two given points of $\Lambda(\Gamma) \simeq \partial_\infty (\Gamma,\Pi)$, and one can thus identify $\mathcal{G}(\mathcal C(\Gamma))$ with $\mathcal{F}(\Gamma,\Pi)$ by an isomorphism of $\Gamma$-flows. In particular, the inclusion $\Gamma \hookrightarrow \SL(2,\Cb)$ is also $1$-Anosov in restriction to $\mathcal F(\Gamma,\Pi)$.

\begin{remark}
  While the relatively property is independent of the Gromov model, Zhu and Zimmer do not state that it implies the Anosov property in restriction to $\mathcal F(\Gamma,\Pi)$. This will be proven by the second author in a forthcoming paper. Here, we do not care about this subtlety because we will ultimately consider geometrically finite subgroups of $\PSL(2,\Cb)$, for $\mathcal F(\Gamma,\Pi)$ is isomorphic to the geodesic flow of the convex core in $\Hb^3$.
\end{remark}

\subsection{Simple Anosov representations}\label{sadef}
An example that motivates the study of restricted Anosov representations is the notion of primitive-stable representations introduced by Minsky \cite{Min}. Let $F_n$ be a free group of order $n$. Let $\Fc_{Prim}\subset\mathcal{F}(F_n)$ denote the \emph{primitive geodesic flow}, which is the closure of the collection of all geodesic axes of primitive elements of $F_n$. The second author proved in \cite{W23} Section 8.1 that a representation $\rho:F_n \to \SL(d,\Kb)$ is $k$-primitive-stable in the sense of Minsky \cite{Min}, Guéritaud--Guichard--Kassel--Wienhard \cite{GGKW} and Kim--Kim \cite{KK}, if and only if $\rho$ is $k$-Anosov in restriction to $\Fc_{Prim}$.

Here we introduce the notion of \emph{simple Anosov representations} which can be thought of as an analogue of primitive-stable representations for closed surface groups.

Let $\pi_1(S)$ be the fundamental group of a closed connected oriented surface $S$ of genus $g\geq 2$. Such a surface carries hyperbolic metrics, and for each hyperbolic metric $h$, there is a discrete and faithful representation $j: \pi_1(S) \to \PSL(2,\Rb)$ (a \emph{Fuchsian representation}) such that $(S,h)$ is isomorphic to $j(\pi_1(S))\setminus \Hb^2$.

Since $\Hb^2$ is uniquely geodesic with the $\pi_1(S)$-action cocompact, the $\pi_1(S)$-flow $\mathcal F(\pi_1(S))$ is isomorphic to the geodesic flow $\psi$ on the unit tangent bundle $T^1(\Hb^2)$ equipped with the $\pi_1(S)$-action given by $\rho$.\footnote{Concretely, one can fix a base point $x_0\in \Hb^2$ and parametrize each oriented geodesic in $\mathbb{H}^2$ by length, with the projection of the base point at $0$. This gives flow-equivariant homomorphism from $\partial_\infty \pi_1(S)^{(2)}$ to $T^1(\Hb^2)$.}

 Therefore, by Proposition \ref{anosovthroughmorphism} and \cite{BPS} Section 4, a representation $\rho:\pi_1(S)\to \SL(d,\Kb)$ is $k$-Anosov (in the classical sense) if and only if $\rho$ is $k$-Anosov in restriction to $(T^1(\Hb^2),\psi)$.

Now, let $\mathcal{F}_p(S) \subset T^1(S)$ be the closure of the union of all closed geodesics with at most $p$ self-intersections, and denote by $\mathcal{F}_p(\pi_1(S))$ its preimage in $T^1(\Hb^2) \simeq \mathcal F(\pi_1(S))$. Then $\mathcal{F}_p(\pi_1(S))$ is again a cocompact $\pi_1(S)$-subflow of $\mathcal F(\pi_1(S))$.

\begin{proposition} \label{p-SelfIntersections}
    Let $c \subset T^1(S)$ be a closed geodesic with more than $p$ self-intersections. Then $c\subset T^1(S) \setminus \mathcal F_p(S)$.
\end{proposition}

\begin{proof}
    Let $v$ be a point in $c$ and let $T$ be the first positive time such that $\psi_T(v) = v$. By assumption, there exist $p+1$ pairs of times $(t_i, t'_i) \in [0,T)^2$ with $t_i \neq t_i'$ such that $\psi_{t_i}(v)$ and $\psi_{t'_i}(v)$ have the same projection to $S$. Up to replacing the initial vector $v$ by another point on $c$, we can assume that none of the $t_i$ is $0$. The corresponding self-intersection in $S$ is transverse and thus stable by small perturbation. In particular, if $v_n$ converges to $v$, then for $n$ large enough, one can find times $(t_{i,n}, t'_{i,n}) \in (0,T)^2, 1\leqslant i\leqslant p+1$ such that $\psi_{t_{i,n}(v_n)}$ and $\psi_{t'_{i,n}}(v_n)$ have the same projection to $S$. In other words, for $n$ large enough, the geodesic arc $\psi_{(0,T)}(v_n)$ has at least $p+1$ self-intersections. In particular, $v$ cannot be approximated by points belonging to a closed geodesic with at most $p$ self-intersections, showing that $c\subset T^1(S) \setminus \mathcal F_p(S)$.
\end{proof}

Note that $\mathcal F_p(S) \subset \mathcal F_q(S)$ for $p\leqslant q$. In particular, all these subflows contain $\mathcal{F}_0(S)$, the closure of the union of all simple closed geodesics in $S$. This set has been studied by Birman and Series \cite{BS85}, who proved in particular that it has Hausdorff dimension $1$. It is thus a very ``small'' subset of the geodesic flow of $S$. We call it the \emph{Birman--Series set} of $S$.

\begin{definition}
    A representation $\rho:\pi_1(S)\to \SL(d,\Kb)$ is called \emph{simple $k$-Anosov} if $\rho$ is $k$-Anosov in restriction to $\mathcal{F}_0(\pi_1(S))$.
\end{definition}

\section{Construction of simple Anosov representations}\label{ConstructionSection}

We can now state more precisely the main result of the paper:
\begin{theorem} \label{thm: ExistenceSimpleAnosov}
    For every $p\geq 0$ and every $d\geq 2$, there exists a representation $\rho: \Gamma \to \SL(2d,\Cb)$ that is $d$-Anosov in restriction to $\mathcal F_p(\Gamma)$ but not in restriction to $\mathcal F_{p+1}(\Gamma)$.
\end{theorem}

By stability of the restricted Anosov property, we can deform such a representation in order to guarantee further generic properties. In particular, we have the following:
\begin{corollary} \label{c-GenericPrimitiveStable}
    For every $d\geq 2$, there exists a non-empty open set of $\Hom(\Gamma, \SL(2d,\Cb))$ consisting of Zariski dense representations that are either non-discrete or unfaithful.
\end{corollary}

The rest of the section is devoted to the proof of these results.

\subsection{Geometrically finite representations of surface groups}
Let us start by recalling that there exist geometrically finite representations of closed surface groups with parabolic subgroups given by any prescribed simple closed curve. More precisely, let $S$ be a closed connected oriented surface of genus $g\geq 2$, $c$ a simple closed curve on $S$ and $\Pi = \langle \gamma \rangle$ the cyclic subgroup of $\Gamma = \pi_1(S)$ generated by an element of $\gamma$ representing $c$.

The following proposition is well-known to Kleinian group experts:

\begin{proposition} \label{existsgfrep}
    There exists $\rho: \Gamma \to \SL(2,\Cb)$ discrete and faithful with geometrically finite image and whose stabilizers of parabolic points are exactly the conjugates of $\Pi$.

    In other terms, there exists a relative $1$-Anosov representation of the relatively hyperbolic pair $(\Gamma,\Pi)$ into $\SL(2,\Cb)$.
\end{proposition}

Such representations can be constructed using the Maskit combination theorem for amalgamated products (when $\gamma$ is separating) and HNN extensions (see Theorem 4.104, Theorem 4.105 and Example 4.106 of Kapovich's book \cite{KapovichHyperbolicManifolds}). A priori, the representations obtained take values in $\Isom_+(\Hb^3)\simeq \PSL(2,\Cb)$. However, discrete and faithful representations of surface groups into $\PSL(2,\Cb)$ have vanishing second Stiefel--Whitney class and can thus be lifted to $\SL(2,\Cb)$.

\subsection{The induced representation}
Let us now recall classical construction of an \emph{induced representation} from a finite index subgroup to the whole group.

Let $\Gamma$ be a countable group and $\Gamma'$ be a subgroup of $\Gamma$ of finite index. Let $V$ be a finite dimensional complex vector space and let $\rho:\Gamma'\to \GL(V)$ be a linear representation. Recall that in representation theory of groups, giving such representation $\rho$ is equivalent to equipping $V$ with a structure of $\Cb[\Gamma']$-module. The induced representation $\mathrm{Ind}^{\Gamma}_{\Gamma'}(\rho)$ of $\Gamma$ is the representation associated to the $\Cb[\Gamma]$-module structure of $\Cb[\Gamma]\otimes_{\Cb[\Gamma']} V$. If $\Gamma'$ has index $d$ in $\Gamma$, then $\Cb[\Gamma]$ is a free $\Cb[\Gamma']$-module of rank $d$, hence $\mathrm{Ind}^{\Gamma}_{\Gamma'}(\rho)$ is a representation of rank $d\cdot \mathrm{dim}_{\Cb}(V)$.

In more concrete terms, pick a collection $\{\gamma_1=id, \gamma_2, ..., \gamma_d\}\subset \Gamma$ of representatives of left cosets of $\Gamma'$, so that \[\Gamma=\bigsqcup_{i=1}^d\gamma_i\Gamma'~.\] 
The $\Cb[\Gamma]$-module $\Cb[\Gamma]\otimes_{\Cb[\Gamma']} V$ can be identified with $\oplus_{i=1}^d\gamma_iV$, where each $\gamma_i V$ is a copy of $V$. For any $1\leqslant i \leqslant d$, we denote the copy of $v\in V$ in $\gamma_i V$ by $(\gamma_iv)$. Then the induced $\Gamma$-action is defined by $\gamma(\gamma_iv) = (\gamma_j(\gamma'v))$ for any $\gamma\in \Gamma$, where $\gamma_j$ and $\gamma'\in \Gamma'$ are such that $\gamma\gamma_i=\gamma_j\gamma'$.

In particular, when $\Gamma'$ is a normal subgroup in $\Gamma$, the restriction of $\mathrm{Ind}^{\Gamma}_{\Gamma'}(\rho)$ to $\Gamma'$ is precisely $\oplus_{i=1}^d\rho_i$, where $\rho_i$ is the representation of $\Gamma'$ defined by $\rho_i(\gamma) = \rho(\gamma_i^{-1}\gamma \gamma_i)$ for all $\gamma\in \Gamma'$.

\begin{lemma}\label{toinducedrep}
    Let $Y$ be a cocompact $\Gamma$-flow. If $\Gamma'\subset \Gamma$ is a normal subgroup of index $d$ and $\rho: \Gamma' \to \SL(2,\mathbb C)$ is $1$-Anosov in restriction to $Y$, then $\mathrm{Ind}^{\Gamma}_{\Gamma'}(\rho): \Gamma\to \SL(2d,\mathbb C)$ is $d$-Anosov in restriction to $Y$.
\end{lemma}
\begin{proof}
    Since $\rho: \Gamma' \to \SL(2,\mathbb C)$ is $1$-Anosov in restriction to $Y$, there exists a dominated splitting of rank $1$, denoted by $Y\times \Cb^2 = E^s_{\rho} \oplus E^u_{\rho}$, where $E^s_{\rho}$ is the stable direction and $E^u_{\rho}$ is the unstable direction, with respect to a $\rho$-invariant metric $\Vert \cdot \Vert$ on $Y\times \Cb^2$. Following Remark \ref{rescalingnorm}, we may assume that $\Vert \cdot \Vert$ is of unit volume and that $E^s_{\rho}$ and $E^u_{\rho}$ are orthogonal with respect to $\Vert \cdot \Vert$. 
    
    Fix $y\in Y$, $v\in (E^s_{\rho})_y$ and $w\in (E^u_{\rho})_y$. Since both $E^s_{\rho}$ and $E^u_{\rho}$ have rank $1$, the previous conditions on $\Vert \cdot \Vert$ imply that the product $\Vert \phi^t(v)\Vert \cdot \Vert \phi^t(w) \Vert$ is constant, and the dominated splitting condition then gives us constants $C,\lambda >0$ (independent of $y,v,w$), such that 
    
    \begin{equation} \label{eq: improved domination condition}\Vert \phi^t(v)\Vert \leqslant C e^{-\lambda t}\Vert v\Vert\text{ and } \Vert \phi^t(w)\Vert \geqslant C e^{\lambda t}\Vert w\Vert~,\end{equation}
    for all $t\in\Rb_+$. This will imply that the direct sum of several such dominated splittings is again a dominated splitting.

    Let $\{\gamma_1= \id, \gamma_2, \ldots, \gamma_d\}$ be a collection of representatives of the left cosets of $\Gamma'\subset \Gamma$. For each $i\in \{1,\ldots ,d\}$, we define the subbundles $E^s_i$ and $E^u_i$ of $Y\times \Cb^2$ by \[(E^s_i)_y=(E^s_{\rho})_{\gamma_i^{-1}y}~, \quad (E^u_i)_y=(E^u_{\rho})_{\gamma_i^{-1}y}~.\]
    
    The splitting $E^s_i\oplus E^u_i$ is the pull-back of the splitting $E^s_{\rho}\oplus E^u_\rho$ by $\gamma_i^{-1}$ acting on $Y$. One can easily show that $Y\times \Cb^2 = E^s_i\oplus E^u_i$ is a $\phi$-invariant and $\rho_i$-equivariant splitting, where $\rho_i(\gamma) = \rho(\gamma_i^{-1}\gamma \gamma_i)$ for any $\gamma\in \Gamma'$.
    Moreover, this splitting is dominated with respect to the pull-back metric ${\gamma_i^{-1}}^*\Vert \cdot \Vert$.
    
   Consider $Y\times \Cb^{2d} = \bigoplus_{i=1}^d Y\times \Cb^2$, the direct sum of $d$ copies of a the trivial rank $2$ bundle, where the $i$. Setting $F^s = \oplus_{i=1}^d E^s_i$ and $F^u = \oplus_{i=1}^d E^u_i$, on obtains a rank $d$ splitting which is equivariant for the representation $\oplus_{i=1}^d \rho_i$. By \eqref{eq: improved domination condition}, this splitting is dominated (for the metric $\oplus_{i=1}^d {\gamma_i}^* \Vert \cdot \Vert$, for instance). Then we conclude that $\oplus_{i=1}^d \rho_i= \Gamma'\to \SL(2d,\Cb)$ is $d$-Anosov in restriction to $Y$. Since $\oplus_{i=1}^d \rho_i$ is the restriction to $\Gamma'$ of $\mathrm{Ind}^{\Gamma}_{\Gamma'}(\rho)$, we conclude that $\mathrm{Ind}^{\Gamma}_{\Gamma'}(\rho)$ is $d$-Anosov in restriction to $Y$ by Proposition \ref{subgptogp}.
\end{proof}

\subsection{Construction of simple Anosov representations}
We now have all the tools to prove Theorem \ref{thm: ExistenceSimpleAnosov}.

Let $S = \Gamma \setminus \Hb^2$ be a closed connected oriented hyperbolic surface. Fix $p\geq 0$ and $d\geq 2$. 

\begin{proposition}
    There exists a Galois covering $\pi:\widetilde{S} \to S$ of degree $d$ and a simple closed geodesic $c \in \widetilde{S}$ such that $\pi(\gamma)$ has $p+1$ self-intersections.
\end{proposition}

An example of such a pair $(\widetilde{S},c)$ is shown in Figure \ref{GaloisCoveringFig1}. 

\begin{figure}
    \include{Figure1.tex}
    \caption{}
    \label{GaloisCoveringFig1}
\end{figure}

    Let $\Gamma'\subset \Gamma$ be the fundamental group of $\widetilde{S}$ and let $\Pi$ be the cyclic subgroup generated by a representative $\gamma$ of $c$ in $\Gamma'$. By Proposition \ref{existsgfrep}, there exists a representation $\rho':\Gamma'\to \SL(2,\Cb)$ which is $1$-Anosov relative to $\Pi$. We set 
    \[\rho = \mathrm{Ind}^\Gamma_{\Gamma'}(\rho'): \Gamma\to \SL(2d,\Cb)~.\]
    
\begin{theorem} \label{thm: ExampleSimpleAnosov}
    The representation $\rho$ is $d$-Anosov in restriction to $\mathcal F_p$ and not $d$-Anosov in restriction to $\mathcal F_{p+1}$.
\end{theorem}

\begin{proof}
    Denote by $Y \subset T^1(\Hb^2)$ the preimage of $T^1(S) \setminus \pi(c)$. It is an open $\Gamma$-subflow of $T^1(\Hb^2)$. 
    
    By Proposition \ref{relofsubflows} and Corollary \ref{relativeimplyflowanosov}, the representation $\rho'$ is $1$-Anosov in restriction to any cocompact subflow of $Y$. By Lemma \ref{toinducedrep}, the induced representation $\rho$ is $d$-Anosov in restriction to any cocompact subflow of $Y$. Finally, by Proposition \ref{p-SelfIntersections}, the curve $c$ is disjoint from $\mathcal F_p(S)$. Hence $\mathcal F_p(\Gamma)$ is a cocompact subflow of $Y$, and $\rho$ is $d$-Anosov in restriction to $\mathcal F_p$.

    On the other hand, $\rho(\gamma)$ is a direct sum of $d$ matrices in $\SL(2,\Cb)$, hence its eigenvalues have the form $\lambda_1, \ldots, \lambda_d, \lambda_d^{-1}, \ldots, \lambda_1^{-1}$, with $\vert \lambda_1\vert \geq \ldots \geq \vert \lambda_d \vert \geq 1$. One of these matrices is $\rho'(\gamma)$ which has eigenvalues $\pm 1$. We deduce that $\vert \lambda_d\vert =1 = \vert \lambda_d^{-1}\vert$ and $\rho(\gamma)$ is not $d$-proximal. Since $\pi(c) \subset \mathcal F_{p+1}(S)$, this implies that $\rho$ is not $d$-Anosov in restriction to $\mathcal F_{p+1}(S)$.

    This concludes the proof of Theorem \ref{thm: ExampleSimpleAnosov}, hence that of Theorem \ref{thm: ExistenceSimpleAnosov}.
\end{proof}

\subsection{Deformations and generic simple Anosov representations}

In this section, we consider small deformations of the above constructed representation to deduce Corollary \ref{c-GenericPrimitiveStable}.

Let $\Hom_{df}(\Gamma,\SL(n,\Cb)) \subset \Hom(\Gamma,\SL(n,\Cb))$ denote the set of discrete and faithful representations and $\Hom_{ndf}(\Gamma,\SL(n,\Cb))$ its complement. As a consequence of the Zassenhaus lemma, $\Hom_{df}(\Gamma,\SL(n,\Cb))$ is closed, hence $\Hom_{ndf}(\Gamma,\SL(n,\Cb))$ is open in $\Hom(\Gamma,\SL(n,\Cb))$.

Let $\Hom^Z(\Gamma,\SL(n,\Cb)) \subset \Hom(\Gamma,\SL(n,\Cb))$ denote the set of representations with Zariski dense image. It is a Zariski open subset of the representation variety which intersects every irreducible component (see for instance \cite{LabourieBook}).

Finally, denote by $\mathrm{An}^d_{\mathcal F_p}(\Gamma, \SL(2d,\Cb))$ the set of representations that are $d$-Anosov in restriction to $\mathcal F_p$, which is open by Proposition \ref{stabilityofRA}. Corollary \ref{c-GenericPrimitiveStable} admits the following reformulation:

\begin{proposition}
    The intersection \[\Hom_{ndf}(\Gamma,\SL(2d,\Cb)) \cap \Hom^Z(\Gamma,\SL(2d,\Cb)) \cap \mathrm{An}^d_{\mathcal F_p}(\Gamma, \SL(2d,\Cb))\]
    is non-empty.
\end{proposition}

\begin{proof}
In the previous section, we constructed a representation $\rho \in \mathrm{An}^d_{\mathcal F_p}(\Gamma, \SL(2d,\Cb)$ of the form $\mathrm{Ind}^\Gamma_{\Gamma'}(\rho')$ where $\Gamma'$ is a subgroup of $\Gamma$ of index $d$ and $\rho':\Gamma'\to \SL(2,\Cb)$ is discrete and faithful but not convex-cocompact. 

By Sullivan's stability theorem for Kleinian groups \cite{Sullivan}, there exists a sequence $\rho'_n \in \Hom_{\textit{ndf}}(\Gamma', \SL(2,\Cb))$ converging to $\rho'$.\footnote{In this precise case, there are more explicit ways to construct the sequence $\rho_n$ than to invoke Sullivan's stability.}

Then $\rho_n \equaldef \mathrm{Ind}^\Gamma_{\Gamma'}(\rho'_n)$ belongs to $\Hom_{\textit{ndf}}(\Gamma, \SL(2d,\Cb))$ and converges to $\rho$. For $n$ large enough, $\rho_n$ is $d$-Anosov in restriction to $\mathcal F_p$ and we conclude that 
\[\Hom_{\textit{ndf}}(\Gamma, \SL(2d,\Cb))\cap \mathrm{An}^d_{\mathcal F_p}(\Gamma, \SL(2d,\Cb))\neq \emptyset~.\]

Finally, since $\Hom^Z(\Gamma, \SL(2d,\Cb))$ is the complement of a subvariety of $\Hom(\Gamma, \SL(2d,\Cb))$ that does not contain an irreducible component, $\Hom^Z(\Gamma, \SL(2d,\Cb))$ intersects every non-empty open subset, hence 
\[\Hom^Z(\Gamma, \SL(2d,\Cb))\cap \Hom_{\textit{ndf}}(\Gamma, \SL(2d,\Cb))\cap \mathrm{An}^d_{\mathcal F_p}(\Gamma, \SL(2d,\Cb))\neq \emptyset~.\]
\end{proof}

\section{Mapping class group dynamics}\label{MCGSection}
In the final section of this paper, we introduce the action of the mapping class group $\Mod(S)$ on geodesic flows and character varieties. We remark that the subflows $\mathcal F_p(S)$ are ``invariant under the Mapping Class Group'', and that $\mathcal F_0(S)$ is the unique minimal subflow with this property. We then deduce that the domains of Anosov representations in restriction to $\mathcal F_p$ form domains of discontinuity for the mapping class group action on character varieties, among which the domains of Anosov representations in restriction to the Birman--Series flow are maximal.

\subsection{Mapping class group invariant closed subflows}
Let $\Gamma$ be a finitely generated hyperbolic group. Recall that every automorphism of $\Gamma$ is a quasi-isometry, and thus extends to a homeomorphism of $\partial_\infty \Gamma$. This defines an action of $\Aut(\Gamma)$ on $\partial_\infty \Gamma$. The restriction of this action to the inner automorphism group $\Inn(\Gamma)$ on $\partial_\infty \Gamma$ is the standard action of $\Gamma$ (i.e. it coincides with the action on the boundary extending left translations on the Cayley graph). 

The action of $\Aut(\Gamma)$ on the boundary at infinity naturally induces an action on $\partial_\infty^{(2)}\Gamma$. Now, every cocompact $\Gamma$-subflow of $\mathcal{F}(\Gamma)$, has the form $Y_P = P\times \Rb$, where $P$ is a closed, $\Gamma$-invariant subset of $\partial_\infty^{(2)}\Gamma$. Given  a subgroup $H$ of $\Out(\Gamma)$ with $\hat H$ its preimage in $\Aut(\Gamma)$, we say that a cocompact $\Gamma$-subflow $Y_P$ is $H$-invariant if $P$ is $\hat{H}$-invariant. This is well-defined since $P$ is always $\Inn(\Gamma)$-invariant by the definition.

\begin{example}
    Recall the notations of Section \ref{sadef},
    \begin{itemize}
        \item for $F_n$, the free group of order $n$, the primitive geodesic flow $\Fc_{Prim}$ is an $\Out(F_n)$-invariant $F_n$-subflow, since the set of primitive elements is $\Out(F_n)$-invariant;
        \item for a compact hyperbolic surface $S$ (possibly with boundary), the flow of geodesics with at most $p$ self-intersections $\mathcal{F}_p(S)$ is a $\Mod(S)$-invariant $\pi_1(S)$-subflow, where $\Mod(S)$ denotes the mapping class group of $S$.
\end{itemize}
\end{example}

In particular, the Birman--Series flow is $\Mod(S)$-invariant. Here we prove that it is the unique minimal $\Mod(S)$-invariant subflow.

\begin{theorem}
    Let $S$ be a closed connected hyperbolic surface, $H$ a finite index subgroup of $\Mod(S)$ and $Y$ a (non-empty) $H$-invariant cocompact subflow of $\mathcal F(\pi_1(S))$. Then $Y$ contains the Birman--Series set $\mathcal F_0(\pi_1(S))$.
\end{theorem}

\begin{proof}
    Let $\ell$ be a geodesic contained in $Y$ and let $c$ be a simple closed geodesic that intersects with $\ell$. Let $T_c$ denote the Dehn twist along $c$. Then there exists $k>0$ such that $T_c^k \in H$. Since $Y$ is closed and $H$-invariant, it contains 
    \[\bigcup_{n\in \Zb}T_c^{nk}(\ell) \supset c~.\]
    We conclude that $Y$ contains any simple closed geodesic that intersects it. In particular it contains a simple closed curve $c_0$. Let $c$ be any simple closed geodesic. There exists a simple closed geodesic $c'$ that intersects both $c$ and $c_0$. By the preceding argument, $Y$ contains $c'$, hence it contains $c$. We conclude that $Y$ contains every simple closed curve, and thus contains the Birman--Series set $\mathcal{F}_0(S)$.
\end{proof}

\subsection{Action on character varieties}

Recall that for a $\Gamma$-flow $(Y,\phi)$, $\mathrm{An}^k_Y(\Gamma, \SL(d,\Kb)) \subset \Hom(\Gamma,\SL(d,\Kb))$ denotes the collection of representations of $\Gamma$ into $\SL(d,\Kb)$ that are $k$-Anosov in restriction to $Y$.
We denote the collection of such representations modulo conjugations by $\mathcal{A}^k_Y(\Gamma, \SL(d,\Kb)) = \mathrm{An}^k_Y(\Gamma, \SL(d,\Kb))\big/ \SL(d,\Kb) \subset \chi(\Gamma, \SL(d,\Kb))$.

Now we are ready to finish the proof of Proposition \ref{p-MCGProperlyDiscontinuous}.

\begin{theorem}\label{simpleAnosovprodis}
    Let $S$ be a closed connected hyperbolic surface.  
 Then the set $\mathcal{A}^d_{\mathcal{F}_p(S)}(\pi_1(S),\SL(2d,\Cb))$ is open in $\chi(\pi_1(S),\SL(2d,\Cb))$, and $\Mod(S)$ acts properly discontinuously on it.
\end{theorem}

We will need the following theorem for the proof.

\begin{theorem}[\cite{W23} Theorem 1.1 and Remark 7.5]\label{prevresult}
    Let $\mathcal{S}$ be a finite, symmetric generating set of $\Gamma$ and $\lambda\geqslant 1$, $\epsilon,b\geqslant 0$ are any given constants. Let
    \[Q_{P,b}=\{ \ell:\mathbb{R}\rightarrow \Cay(\Gamma,\mathcal{S})\ |\ \ell\ \text{is a geodesic with}\ d(\ell(0),\id)\leqslant b\ \text{and}\ (\ell(-\infty),\ell(+\infty))\in P \}\] and \[\Gamma^+_{P,b}=\{ \ell(t)\ |\ \ell\in Q_{P,b}, t\in \mathbb{R}_+\ \text{with}\ \ell(t)\in \Gamma \text{ a vertex of }\Cay(\Gamma,\mathcal{S}) \}\subset \Gamma~.\]
    Let $O$ be an open, relatively compact subset in $A^k_{S_P}(\Gamma, \SL(d,\Kb))$. Then there exist constants $A\geqslant 1$ and $B\geqslant 0$, such that \[\log\dfrac{\sigma_{k}(\rho(\gamma))}{\sigma_{k+1}(\rho(\gamma))}\geqslant A^{-1}\abs{\gamma} - B~,\] for any $\rho\in O$ and $\gamma\in \Gamma^+_{P,b}$, where $\abs{\gamma}$ is the word length of $\gamma$ with respect to $\mathcal{S}$, and $\sigma_k(\rho(\gamma))$ is the $k$\textsuperscript{th} singular value of $\rho(\gamma)$.
\end{theorem}

\begin{proof}[Proof of Theorem \ref{simpleAnosovprodis}]
    Let $\mathcal{S} = \{s_1,s_2,...,s_{n}\}$ be a generating set of $\pi_1(S)$ such that
    \begin{itemize}
        \item Each $s_i$ represents a simple closed geodesic;
        \item For any $i\ne j$, at least one of $s_is_j$ and $s_is_j^{-1}$ represents a simple closed geodesic.
    \end{itemize}
    Let $D$ denote the collection of all $s_i$ and all $s_is_j^{\pm 1}$ that represent simple closed geodesics.
    
    Let $P\subset \partial_\infty^{(2)}\pi_1(S)$ be the closed subset such that $S_P = \mathcal{F}_p(S)$. Since $A^d_{\mathcal{F}_p(S)}(\Gamma, \SL(2d,\Cb))$ is open in $\Hom(\Gamma, \SL(2d,\Cb))$ (by Proposition \ref{stabilityofRA}) and conjugation invariant, we have $\mathcal{A}^d_{\mathcal{F}_p(S)}(\Gamma, \SL(2d,\Cb))$ is open in $\chi(\Gamma, \SL(2d,\Cb))$.

    Let $\ell$ be a compact subset of $\mathcal{A}^d_{\mathcal{F}_p(S)}(\pi_1(S),\SL(2d,\Cb))$. We pick $b\geqslant 0$ large enough such that all powers of elements in $D$ contains in $ \Gamma^+_{P,b}$. By Theorem \ref{prevresult}, for any $\rho_0\in A^d_{\mathcal{F}_p(S)}(\Gamma, \SL(2d,\Cb))$ with $[\rho_0]\in L$, there exists a relatively compact neighborhood $O$ and constants $A_O\geqslant 1$ and $B_O\geqslant 0$, such that \[ A_O^{-1}\abs{\gamma}- B_O\leqslant\log\dfrac{\sigma_{d}(\rho(\gamma))}{\sigma_{d+1}(\rho(\gamma))}\leqslant \log\dfrac{\sigma_{1}(\rho(\gamma))}{\sigma_{2d}(\rho(\gamma))}\leqslant A_O\abs{\gamma}~,\]
    for any $\rho\in O$ and $\gamma\in \Gamma^+_{P,b}$. Here the third inequality follows from that $\pi_1(S)$ is finitely generated. Then for any $\gamma\in D$, we have \[A_O^{-1}\Vert\gamma\Vert\leqslant\log\dfrac{|\lambda_{d}(\rho(\gamma))|}{|\lambda_{d+1}(\rho(\gamma))|} = \lim\limits_{n\to \infty}\dfrac{1}{n} \log\dfrac{\sigma_{d}(\rho(\gamma^n))}{\sigma_{d+1}(\rho(\gamma^n))} \leqslant A_O\Vert\gamma\Vert~,\] where $\Vert\gamma\Vert = \lim\limits_{n\to \infty} \dfrac{\abs{\gamma}^n}{n}$ is the stable length of $\gamma$ in $\Cay(\pi_1(S),\mathcal{S})$ and $\lambda_k$ denote the $k$\textsuperscript{th} eigenvalue (ranking by absolute values). Since both stable lengths in $\Cay(\pi_1(S),\mathcal{S})$ and eigenvalues are invariant by conjugation, the inequality holds for all $\rho \in \SL(d,\Cb)\cdot O$. Let $[O]$ denote the conjugation classes of $\SL(d,\Cb)\cdot O$, then $\ell$ is covered by finitely many such open set $[O]$. Therefore, there exists a constant $A$, such that \[A^{-1}\Vert\gamma\Vert\leqslant\log\dfrac{|\lambda_{d}(\rho(\gamma))|}{|\lambda_{d+1}(\rho(\gamma))|} \leqslant A\Vert\gamma\Vert~,\] for all $\gamma\in D$.

    Suppose there exists $[f]\in\Out(\pi_1(S))$ with a representative $f\in\Aut(\pi_1(S))$, such that there is a representation $[\rho]\in L$ with $[\rho\circ f]\in L$, where $\rho$ is a representative of $[\rho]$. Then we have \[\Vert f(\gamma)\Vert \leqslant A^2 \Vert\gamma\Vert~,\] for any $\gamma\in D$.

    Therefore, to show the $\Mod(S)$-action is properly discontinuous, it suffices to check \[\{[f]\in \Out(\pi_1(S))\ |\ \Vert f(\gamma)\Vert \leqslant A^2 \Vert\gamma\Vert\ \text{for all}\ \gamma\in D \}\] is finite. This immediately follows from the proof of Lemma 12 in \cite{Lee11}.
\end{proof}

\begin{remark}
    It provides a sequence of domains of discontinuous \[\mathcal{A}^d_{\mathcal{F}_{p+1}(S)}(\pi_1(S),\SL(2d,\Cb))\subset \mathcal{A}^d_{\mathcal{F}_{p}(S)}(\pi_1(S),\SL(2d,\Cb))~,\] all of which are contained in $\mathcal{A}^d_{\mathcal{F}_0(S)}(\pi_1(S),\SL(2d,\Cb))$.
\end{remark}

\bibliographystyle{alpha}
\bibliography{reference.bib}

\begin{thebibliography}{GGKW17}

\bibitem[Bow12]{Bow12}
Brian~H Bowditch.
\newblock Relatively hyperbolic groups.
\newblock {\em International Journal of Algebra and Computation},
  22(03):1250016, 2012.

\bibitem[BPS19]{BPS}
Jairo Bochi, Rafael Potrie, and Andr{\'e}s Sambarino.
\newblock Anosov representations and dominated splittings.
\newblock {\em Journal of the European Mathematical Society},
  21(11):3343--3414, 2019.

\bibitem[BS85]{BS85}
Joan~S Birman and Caroline Series.
\newblock Geodesics with bounded intersection number on surfaces are sparsely
  distributed.
\newblock {\em Topology}, 24(2):217--225, 1985.

\bibitem[CZZ22]{CZZ}
Richard Canary, Tengren Zhang, and Andrew Zimmer.
\newblock Cusped {H}itchin representations and {A}nosov representations of
  geometrically finite {F}uchsian groups.
\newblock {\em Advances in Mathematics}, 404:108439, 2022.

\bibitem[GGKW17]{GGKW}
Fran{\c{c}}ois Gu{\'e}ritaud, Olivier Guichard, Fanny Kassel, and Anna
  Wienhard.
\newblock Anosov representations and proper actions.
\newblock {\em Geometry \& Topology}, 21(1):485--584, 2017.

\bibitem[GM08]{GM08}
Daniel Groves and Jason~Fox Manning.
\newblock Dehn filling in relatively hyperbolic groups.
\newblock {\em Israel Journal of Mathematics}, 168(1):317--429, 2008.

\bibitem[Gro87]{Gromov}
Mikhael Gromov.
\newblock Hyperbolic groups.
\newblock In {\em Essays in group theory}, pages 75--263. Springer, 1987.

\bibitem[Kap01]{KapovichHyperbolicManifolds}
Michael Kapovich.
\newblock {\em Hyperbolic manifolds and discrete groups}, volume 183.
\newblock Springer Science \& Business Media, 2001.

\bibitem[KK21]{KK}
Inkang Kim and Sungwoon Kim.
\newblock Primitive stable representations in higher rank semisimple {L}ie
  groups.
\newblock {\em Revista Matem{\'a}tica Complutense}, 34(3):715--745, 2021.

\bibitem[KL23]{KL}
Michael Kapovich and Bernhard Leeb.
\newblock Relativizing characterizations of {A}nosov subgroups, {I} (with an
  appendix by {G}regory {A}. {S}oifer).
\newblock {\em Groups, Geometry, and Dynamics}, 2023.

\bibitem[Lab06]{Lab}
Francois Labourie.
\newblock Anosov flows, surface groups and curves in projective space.
\newblock {\em Inventiones mathematicae}, 165:51--114, 2006.

\bibitem[Lab08]{LabourieMCGAction}
Fran{\c{c}}ois Labourie.
\newblock Cross ratios, anosov representations and the energy functional on
  teichm{\"u}ller space.
\newblock In {\em Annales scientifiques de l'Ecole normale sup{\'e}rieure},
  volume~41, pages 439--471, 2008.

\bibitem[Lab17]{LabourieBook}
Fran{\c{c}}ois Labourie.
\newblock Lectures on representations of surface groups.
\newblock {\em Zurich Lectures in Advanced Mathematics, European Mathematical
  Society (EMS), Z{\"u}rich}, 2017.

\bibitem[Lee15]{Lee11}
Michelle Lee.
\newblock Dynamics on the $\mathrm{PSL}(2,\mathbb{C})$ -character variety of a
  twisted {I}-bundle.
\newblock {\em Groups, Geometry, and Dynamics}, 9(1):187--201, 2015.

\bibitem[Man15]{Man15}
Jason~Fox Manning.
\newblock The {B}owditch boundary of $\partial({G}, \mathcal{{H}})$ when ${G}$
  is hyperbolic.
\newblock {\em arXiv preprint arXiv:1504.03630}, 2015.

\bibitem[Min05]{Mineyev}
Igor Mineyev.
\newblock Flows and joins of metric spaces.
\newblock {\em Geometry \& Topology}, 9(1):403--482, 2005.

\bibitem[Min13]{Min}
Yair~N Minsky.
\newblock On dynamics of ${O}ut({F}_n)$ on $\mathrm{PSL}_2(\mathbb{C})$
  characters.
\newblock {\em Israel Journal of Mathematics}, 193(1):47--70, 2013.

\bibitem[MMMZ23]{MMMZ}
Sara Maloni, Giuseppe Martone, Filippo Mazzoli, and Tengren Zhang.
\newblock $d$-pleated surfaces and their shear-bend coordinates.
\newblock {\em arXiv preprint arXiv:2305.11780}, 2023.

\bibitem[Osi06]{Osin06}
Denis~V Osin.
\newblock {\em Relatively Hyperbolic Groups: Intrinsic Geometry, Algebraic
  Properties, and Algorithmic Problems: Intrinsic Geometry, Algebraic
  Properties, and Algorithmic Problems}, volume 843.
\newblock American Mathematical Soc., 2006.

\bibitem[Shu13]{Shub}
Michael Shub.
\newblock {\em Global stability of dynamical systems}.
\newblock Springer Science \& Business Media, 2013.

\bibitem[SS06]{SS06}
Juan Souto and Peter Storm.
\newblock Dynamics of the mapping class group action on the variety of
  $\mathrm{PSL}_2(\mathbb{C})$ characters.
\newblock {\em Geometry \& Topology}, 10(2):715--736, 2006.

\bibitem[Sul86]{Sullivan}
Dennis Sullivan.
\newblock Quasiconformal homeomorphisms and dynamics, {II}: Structural
  stability implies hyperbolicity.
\newblock {\em Acta Math}, 155:244--260, 1986.

\bibitem[Tra13]{Tran13}
Hung~Cong Tran.
\newblock Relations between various boundaries of relatively hyperbolic groups.
\newblock {\em International Journal of Algebra and Computation},
  23(07):1551--1572, 2013.

\bibitem[Wan23]{W23}
Tianqi Wang.
\newblock Anosov representations over closed subflows.
\newblock {\em Transactions of the American Mathematical Society}, 2023.

\bibitem[Yam04]{Yam}
Asli Yaman.
\newblock A topological characterisation of relatively hyperbolic groups.
\newblock {\em Journal für die reine und angewandte Mathematik}, 2004.

\bibitem[Zhu21]{Zhu}
Feng Zhu.
\newblock Relatively dominated representations.
\newblock {\em Annales de l'Institut Fourier}, 71(5):2169--2235, 2021.

\bibitem[ZZ22a]{ZZ1}
Feng Zhu and Andrew Zimmer.
\newblock Relatively {A}nosov representations via flows {I}: theory.
\newblock {\em arXiv preprint arXiv:2207.14737}, 2022.

\bibitem[ZZ22b]{ZZ2}
Feng Zhu and Andrew Zimmer.
\newblock Relatively {A}nosov representations via flows {II}: examples.
\newblock {\em arXiv preprint arXiv:2207.14738}, 2022.

\end{thebibliography}

\end{document}